\documentclass[12pt,a4paper]{amsart}

\usepackage[T1]{fontenc}
\usepackage[utf8]{inputenc}

\usepackage{commath}
\usepackage{mathtools}
\usepackage{graphicx}
\usepackage{enumitem}
\usepackage{pb-diagram}
\usepackage{amsfonts,amssymb,amsthm,amscd,amsmath}
\usepackage{latexsym,stmaryrd}
\usepackage{lscape}
\usepackage[all]{xy}
\usepackage{setspace}
\usepackage[margin=1.1in]{geometry}
\usepackage{geometry}
\usepackage{bussproofs}
\usepackage[colorlinks,linkcolor=blue,citecolor=magenta]{hyperref}
\usepackage{hyperref}

\theoremstyle{plain}
\newtheorem{theorem}{Theorem}
\newtheorem{lemma}[theorem]{Lemma}
\newtheorem{proposition}[theorem]{Proposition}

\newtheorem{conjecture}[theorem]{Conjecture}
\newtheorem{problem}[theorem]{Problem}

\theoremstyle{definition}

\newtheorem{remark}[theorem]{Remark}

\newcommand{\BB}{\mathbf{B}}
\newcommand{\spinc}{\text{Spin}^{c}}
\newcommand{\es}{\sigma_{e}}

\newcommand{\lang}{\langle}
\newcommand{\rang}{\rangle}

\newcommand{\im}{\mathrm{Im}}
\newcommand{\w}{\wedge}
\newcommand{\pa}{\partial}

\newcommand{\al}{\alpha}
\newcommand{\be}{\beta}

\newcommand{\om}{\omega}
\newcommand{\Om}{\Omega}
\newcommand{\vp}{\varphi}

\newcommand{\bB}{\mathbb{B}}
\newcommand{\bC}{\mathbb{C}}

\newcommand{\bN}{\mathbb{N}}

\newcommand{\bR}{\mathbb{R}}
\newcommand{\bS}{\mathbb{S}}

\newcommand{\bZ}{\mathbb{Z}}
\newcommand{\cA}{\mathcal{A}}

\newcommand{\cH}{\mathcal{H}}

\newcommand{\cM}{\mathcal{M}}
\newcommand{\cN}{\mathcal{N}}

\newcommand{\cQ}{\mathcal{Q}}

\newcommand{\fb}{\mathfrak{b}}

\newcommand{\fj}{\mathfrak{j}}

\newcommand{\fK}{\mathfrak{K}}

\newcommand{\fn}{\mathfrak{n}}

\newcommand{\fs}{\mathfrak{s}}
\newcommand{\fT}{\mathfrak{T}}

\newcommand{\ra}{\rightarrow}
\newcommand{\sub}{\subseteq}

\begin{document}
\title{An index theorem for quotients of Bergman spaces on egg domains}

\author{Mohammad Jabbari}
\address{Mohammad Jabbari, Centro de Investigacion en Matematicas, A.P. 402, Guanajuato, Gto., C.P. 36000, Mexico}
\email{mohammad.jabbari@cimat.mx}

\author{Xiang Tang}
\address{Xiang Tang, Department of Mathematics and Statistics, Washington University in St. Louis, St. Louis, Missouri, 63130, USA}
\email{xtang@wustl.edu}

\maketitle

\begin{center}
Dedicated to the memory of Ronald G. Douglas
\end{center}

\begin{abstract}
In this paper we prove a $K$-homology index theorem for the Toeplitz operators obtained from the multishifts of the Bergman space on several classes of egg-like domains.
This generalizes our theorem with Douglas and Yu on the unit ball  \cite{djty}.
\end{abstract}


\section{Introduction}\label{introduction}
Around a decade ago a multivariate operator theory approach to algebraic geometry was suggested by Arveson and Douglas in the following way \cite{arveson-standard,douglas-index}.
Suppose that $I\sub A:=\bC[z_1,\ldots,z_m]$ is an ideal of the ring of polynomials in $m$ variables.
To understand the geometry of the zero variety $V(I):=\{p\in\bC^m:f(p)=0,\forall f\in I \}$ defined by $I$, algebraic geometers study the coordinate ring $A/I$.
To find an operator theory model for $A/I$, one can replace $A$ by the Bergman space $L^2_a(\Om)$ of square-integrable analytic functions on some bounded strongly pseudoconvex domain $\Om\sub\bC^m$ with smooth boundary, and mod it out by the closure $\overline{I}$ of $I$ inside $L^2_a(\Om)$.
The quotient Hilbert space $\cQ_{I}:=L^2_a(\Om)/\overline{I}$ has a natural Hilbert $A$-module structure\footnote{
There is a one-to-one correspondence between commuting $m$-tuples of operators $T:=(T_1,\ldots,T_m)$ acting on a Hilbert space $\cH$ and Hilbert $A$-module structures on $\cH$ \cite{arveson-standard}. 
The correspondence is given by representing each polynomial $p(z_1,\ldots,z_m)\in A$ by the operator $p(T_1,\ldots,T_m)$. 
Conversely, $T$ is identified with the $m$-tuple $(M_{z_1},\ldots,M_{z_m})$ of multiplication operators by coordinate functions, 
and is called the fundamental tuple of Toeplitz operators on the Hilbert $A$-module $\cH$.
Based on this correspondence, the properties of $T$ are attributed to $\cH$ and vice versa. 
For example,  $\cH$ is called essentially normal (respectively $p$-essentially normal) if all $[T_j,T_k^{\ast}]$ are compact (respectively Schatten $p$-summable). Also, $\es(\cH)$ denotes the essential Taylor spectrum associated to the fundamental tuple of Toeplitz operators of $\cH$ \cite{taylor,muller}.} given by $p\cdot(f+\overline{I})=pf+\overline{I}$, $p\in A$, $f\in L^2_a(\Om)$.
Transporting this action to the orthogonal complement
\[
L^2_a(\Om)\ominus\overline{I}=I^{\perp}\cong\cQ_I,
\]
makes $I^{\perp}$ a Hilbert $A$-module.\label{hilbertmodule}
Alternatively, the module structure of $I^\perp$ is given by the compression of multiplication operators:
\begin{equation*}
T_p:=P_{I^{\perp}}M_p|_{I^{\perp}},\quad p\in A,\label{cs}
\end{equation*}
where $M_p:L^2_a(\Om)\ra L^2_a(\Om)$ is the multiplication by $p$, and $P_{I^{\perp}}$ is the orthogonal projection in $L^2_a(\Om)$ onto $I^\perp$.
Let $\fT_{I}$ be the unital C*-algebra generated by $\{T_p:p\in A\}\cup\fK$, where $\fK$ is the ideal of compact operators on $I^{\perp}$.
Arveson, based on his work on the model theory of spherical contractions in multivariate dilation theory \cite{arveson-dilation},\cite[Chapters 40--41]{alpay}, conjectured that \cite{arveson-dirac,arveson-conjecture}:

\begin{conjecture}[Arveson]\label{conjecture-a}
$I^{\perp}$ is essentially normal.
In other words, all commutators $[T_{z_j},T_{z_k}^{*}]$, $j,k=1,\ldots,m$ are compact.
\end{conjecture}

Suppose momentarily that this conjecture holds. 
Also, assume that $I$ is homogeneous.
Then the maximal ideal space of $\fT_I/\fK$ is homeomorphic via the mapping $\vp\mapsto \left(\vp(T_{z_1}),\ldots,\vp(T_{z_m})\right)$ to the essential Taylor spectrum of $\left(T_{z_1},\ldots,T_{z_m}\right)$, which is itself contained in $X_I:=V(I)\cap\pa\Om$ \cite[Theorem 4.1]{douglas-index}. 
(Also see \cite[Corollary 3.10]{curto},\cite{grs2},\cite[Theorem 5.1]{gw}.)
The Gelfand-Naimark duality then gives the following short exact sequence of C*-algebras:
\begin{equation*}
0\ra \fK\hookrightarrow \fT_I\ra C(X_I)\ra 0.\label{extension}
\end{equation*}
Let 
\[
\tau_I:=[\fT_I]
\]
be the equivalence class represented by this exact sequence in the odd $K$-homology group $K_1\left(X_I\right)$ of Brown-Douglas-Fillmore \cite{bdf1,bdf2}.
Douglas \cite{douglas-index} (see also \cite[Section 25]{bd}) asked for an explicit computation of this element in other topological or geometric realizations of $K$-homology:

\begin{problem}[Douglas]\label{problem-d}
Assume that $I$ is homogeneous and $I^{\perp}$ is essentially normal. 
	Identify $\tau_I\in K_1\left(X_I\right)$.
\end{problem} 

More specifically, in the same paper he conjectured that:

\begin{conjecture}[Douglas]\label{conjecture-d}
Let $I$ be the vanishing ideal of an algebraic variety $V\sub\bC^{m}$ which intersects $\pa\Om$ transversally. 
Then $I^{\perp}$ is essentially normal, and its induced extension class $\tau_I$ is identified with the fundamental class of $X_I$, namely the extension class induced by the $\spinc$ Dirac operator associated to the natural Cauchy-Riemann structure of $X_I$.
\end{conjecture}

By analogy with the Atiyah-Singer index theorem, one expects that this conjecture would lead to new connections between geometry and operator theory.
To see what brought Arveson and Douglas to their conjecture/problem, we refer the reader to their original papers \cite{arveson-conjecture,arveson-curvature,douglas-index} as well as \cite[Chapter 41]{alpay},\cite[Sections 1.2--3]{jabbari-thesis}. 
Specially, Conjecture \ref{conjecture-d} generalizes some aspects of the Boutet de Monvel index theorem for Toeplitz operators on strongly pseudoconvex domains to possibly singular algebraic varieties \cite{b,bdt}.

Let us review some results about these conjectures and problem. 
(See also \cite[Chapter 41]{alpay},\cite{gw-survey}.)
When $\Om$ is the unit open ball, Conjecture \ref{conjecture-a} has been proved for the following cases:
\begin{enumerate}
	\item
	$I$ is monomial \cite{arveson-conjecture, douglas-monomial,djty}.
	
		\item 
		$I$ is homogeneous and	$m\leq 3$ \cite{gw}.

		\item
		$I$ is homogeneous and $\dim V(I)\leq 1$ \cite{gw}.

	\item
	$I$ is principal \cite{gw, dw-harmonic, fx, fx2, dgw, fx3}.
	(The last two references allow for strongly pseudoconvex domains $\Om$.)

	\item
	$I$ has a stable generating set $\{p_1,\ldots,p_k\}$ of homogeneous polynomials in the sense that there exists $C>0$ such that every $q\in I$ can be written as $q=\sum_{j=1}^k r_jp_j$ with $r_j\in A$ and $\|r_jp_j\|_{L^2(\Om)}\leq C\|q\|_{L^2(\Om)}$ \cite{shalit,wang-unified}.

	\item
	$I$ is the vanishing ideal of a homogeneous variety smooth away from the origin \cite{ee,dty, dw-geometric,wx}.
	
\end{enumerate}

When $\Omega$ is the unit ball, the articles \cite{gw} and \cite{dty} answer Problem \ref{problem-d} respectively when $m\leq 2$ and when $I$ is the vanishing ideal of a complete intersection variety (possibly singular away from the boundary).
In \cite{djty} we gave an answer to Problem \ref{problem-d} when $\Om$ is the unit open ball and $I$ is monomial:

\begin{theorem}\label{theorem1}
Let $\Omega$ be the unit open ball $\bB_m$, and $I$ be a monomial ideal.

(a)
There exist a positive integer $k$, essentially normal Hilbert $A$-modules $\cA_0:=L^2_a(\Om)$, $\cA_1,\ldots,\cA_{k}$, and Hilbert $A$-module morphisms\footnote{Bounded linear maps that preserve $A$-module structures.} $\Psi_q:\cA_q\ra\cA_{q+1}$, $q=0,\ldots,k-1$ such that
\begin{equation*}
0\ra \overline{I}\hookrightarrow \cA_0\stackrel{\Psi_0}{\ra} \cA_1\stackrel{\Psi_1}{\ra}\cdots \stackrel{\Psi_{k-1}}{\ra }\cA_k\ra 0\label{resolution}
\end{equation*}
is exact.
(This implies that $I^\perp$ is essentially normal.)

(b)
For each $q$, let $\fT(\cA_q)$ be the unital C*-algebra generated by all module action operators as well as all compact operators on the Hilbert module $\cA_q$, and let $\es^q:=\es^q(\cA_q)$ be the essential Taylor spectrum associated to
$\cA_q$.
Then the identification 
\begin{equation*}
\tau_I=\sum_{q=1}^{k}(-1)^{q-1} [\fT(\cA_q)]\label{equ}
\end{equation*}
holds in $K_1\left(\es^1\cup\cdots\cup\es^k\right)$. \end{theorem}

By its explicit construction, each $\cA_q$ has a tractable geometry as the Hilbert space of square-integrable analytic sections of a Hermitian vector bundle on a disjoint union of subsets of $\bB_m$.

In this paper, we generalize Theorem \ref{theorem1} to the case when $\Omega$ is an egg domain of the form
\begin{equation}
\Om_1:=\left\{\left(z_1,\ldots,z_m\right)\in\bC^m:\sum\limits_{j=1}^{m}\left|z_j\right|^{2p_j}<1\right\},\quad p_j>0,\label{egg}
\end{equation}
or more generally of the form
\begin{equation}
\Om_2:=\left\{\left(\sum\limits_{j=1}^{m}\left|z_j\right|^{2p_j}\right)^{a}+\left(\sum\limits_{k=1}^{n}\left|w_k\right|^{2q_k}\right)^{b}+\cdots<1\right\}\sub\bC^{m+n+\cdots},\label{egg2}
\end{equation}
where the finitely many parameters $p_j, q_k, a,b,\ldots$ are arbitrary positive reals.
(When all $p_j,q_k,\ldots$ equal $1$, $\Om_2$ is called a generalized complex ellipsoid in \cite{kkm}.)


\begin{theorem}\label{theorem2}
Let $\Omega$ be a domain of the form (\ref{egg}) or (\ref{egg2}), and $I$ be a monomial ideal.

(a)
There exist a positive integer $k$, essentially normal Hilbert $A$-modules $\cA_0:=L^2_a(\Om)$, $\cA_1,\ldots,\cA_{k}$, and Hilbert $A$-module morphisms $\Psi_q:\cA_q\ra\cA_{q+1}$, $q=0,\ldots,k-1$ such that
\begin{equation}
0\ra \overline{I}\hookrightarrow \cA_0\stackrel{\Psi_0}{\ra} \cA_1\stackrel{\Psi_1}{\ra}\cdots \stackrel{\Psi_{k-1}}{\ra }\cA_k\ra 0\label{resolution}
\end{equation}
is exact.
(This implies that $I^\perp$ is essentially normal.)

(b)
For each $q$, let $\fT(\cA_q)$ be the unital C*-algebra generated by all module action operators as well as all compact operators on the Hilbert module $\cA_q$, and let $\es^q:=\es^q(\cA_q)$ be the essential Taylor spectrum associated to $\cA_q$.
Then the identification 
\begin{equation*}
\tau_I=\sum_{q=1}^{k}(-1)^{q-1} [\fT(\cA_q)]\label{equ}
\end{equation*}
holds in $K_1\left(\es^1\cup\cdots\cup\es^k\right)$. \end{theorem}

The explicit construction of the resolution (\ref{resolution}) and the proof of Theorem \ref{theorem2} come in Sections \ref{longexactseq} and \ref{proof}, respectively.
Our proof uses crucially the fact that monomials constitute an orthogonal basis for $L^2_a(\Om)$ if $\Om$ is a domain of type (\ref{egg}) or (\ref{egg2}).
Each $\cA_q$ has a tractable geometry as the Hilbert space of square-integrable analytic sections of a Hermitian vector bundle on a disjoint union of subsets of $\Om$.

\begin{remark}
	When $I$ is homogeneous, the C*-algebra generated by $\{1\}\cup\{T_p:p\in A\}$ is irreducible (it has no proper reducing closed subspace), hence contains $\fK$ if $I^{\perp}$ is essentially normal \cite[Page 923]{gw},\cite[Theorem 5.39]{douglas-bat1}.
\end{remark}

\begin{remark}
One reason why we care about monomial ideals is that a comprehensive understanding of the phenomena appearing in this generically nonradical case may lead to new results beyond the recently established ones about radical ideals \cite{dty,dw-geometric,ee}. 
\end{remark}

\begin{remark}
When all $p_j$ are $\geq 1$ and at least one of them is $>1$, domain (\ref{egg}) is weakly (but not strongly) pseudoconvex and with smooth boundary \cite{dangelo0},\cite[Section 11.5]{krantz}.
(The same is true for domain (\ref{egg2}) if all $a, b,\ldots,2p_j,2q_k,\ldots$ are $\geq 2$ and at least one of $2p_j,2q_k,\ldots$ is $>2$.)
As far as we know, putting the polydiscs aside \cite{wz}, Theorem \ref{theorem2} is the only result which discusses Conjecture \ref{conjecture-a} and Problem \ref{problem-d} on weakly pseudoconvex domains.
\end{remark}

\begin{remark}
Note that a domain of type (\ref{egg2}) is obtained from a domain of type (\ref{egg}) when each $|z_j|$ is replaced by an expression of the form $\sum_{k=1}^{n_j} |z_{jk}|^{p_{jk}}$, where all coordinates $z_{jk}$ are distinct.
Applying this process on a domain of type (\ref{egg2}) and repeating this process finitely many times, give rise to more generalized egg domains. For example, we can get the following one:
\begin{multline*}
\left(\left(|z_{111}|^{p_{111}}+|z_{112}|^{p_{112}}\right)^{p_{11}}+\left(|z_{121}|^{p_{121}}+|z_{122}|^{p_{122}}+|z_{123}|^{p_{123}}\right)^{p_{12}}+|z_{13}|^{p_{13}}\right)^{p_1}+\\
(\cdots)^{p_2}+\cdots<1.
\end{multline*}
The arguments in this paper prove Theorem \ref{theorem2} for all such domains.
\end{remark}

Arveson's statement of his essential normality conjecture was more refined than Conjecture \ref{conjecture-a} in the sense that it addressed the Schatten class membership of commutators \cite{arveson-conjecture,douglas-monomial}.
In this paper, however, we merely focused on the membership of commutators in the ideal of compacts.
The reason is that our proof of Theorem \ref{theorem2}.(b) relies crucially on the usage of the Fuglede-Putnam theorem in the proof of Proposition \ref{fact}.(b,c,d).
Since the Schatten class version of the Fuglede-Putnam theorem is missing \cite{douglas-monomial,shulman}, our result does not determine the Schatten class membership of the commutators for the quotients of Bergman spaces by monomial ideals.
Nevertheless, it is worth pointing out that our computations (not included in this paper) show that the whole Bergman space $L^2_a(\Om_1)$, associated with domain (\ref{egg}), is $p$-essentially normal exactly when $p$ is strictly larger than
\[
\max\left\{
m,
p_j(m-1): j=1,\ldots,m\right\}.
\]
(See also Remark \ref{schatten}.)
This suggests that the Schatten class property of the commutators may be related to the convexity and geometry of the domain \cite{connes,dv,milnor}. 
We plan to discuss this relation in the future \cite{jabbari-schatten}.

\[\]
\noindent
\textbf{Acknowledgments.}
We would like to thank Ronald Douglas, Guoliang Yu and Yi Wang for inspiring discussions.
Both authors are partially supported by NSF grants.

%

\section{The construction of the resolution in Theorem \ref{theorem2}}\label{longexactseq}
From now on, $\Om_1$ and $\Om_2$ are domains of type (\ref{egg}) and (\ref{egg2}), respectively.
We develop the details for $\Om_1$, and $\Om_2$ can be treated similarly, with the only difference being Proposition \ref{lemma2} and the proof of Lemma \ref{prop:ess-norm-box}.
We always use the multi-index notation \cite[Page 3]{krantz}, especially $|\al|$ to stand for the sum of the components of the multi-index $\al$.
$\bN$ denotes the set of nonnegative integers.

\subsection{The monomial orthonormal basis for the Bergman space}\label{onb}
Monomial functions $z^\al$, $\al\in\bN^m$ are orthogonal in $L^2_a(\Om_1)$, as the integration in polar coordinates in each variable shows.
On the other hand, since $\Om_1$ is a complete Reinhardt domain, polynomials are dense in $L^2_a(\Om_1)$ with respect to the topology of uniform convergence on compacts \cite[Page 47]{range}.
Then a standard shrinking argument \cite[Page 43]{zhu-FT},\cite[Page 11]{ds} shows that the normalized monomials
\[
\frac{z^\al}{\sqrt{\om_1(\al)}},\quad
\om_1(\al):=\left\|z^\al\right\|^2_{L^2_a(\Om_1)}
\]
constitute an orthonormal basis for the Hilbert space $L^2_a(\Om_1)$. 
Next, we are going to find an explicit formula for $\om_1(\al)$ as well as $\om_{2}(\al,\be,\ldots)
:=\left\|z^\al w^\be\cdots\right\|^2_{L^2_a(\Om_2)}$.  
In what follows, for a variable $x=(x_1,\ldots,x_m)$ ranging on some part of $\bR^m$, $dx:=dx_1\cdots dx_m$ denotes the Riemannian density of the Euclidean space $\bR^m$. 
The set of positive reals is denoted by $\bR_+$.

\begin{lemma}\label{lemma1}
Given $\al:=(\al_1,\ldots,\al_m)\in\bR_+^m$, we have
	\[
	\int_{x\in\bR^m_+,\ \sum x_j^2<1}x^\al dx=\frac{B\left(\frac{\al+1}{2}\right)}{2^m\left|\frac{\al+1}{2}\right|},
	\]
	\[
	\int_{x\in\bR^m_+,\ \sum x_j^2=1}x^\al d\sigma_m(x)=\frac{B\left(\frac{\al+1}{2}\right)}{2^{m-1}},
	\]
	where $\al+1=(\al_1+1,\ldots,\al_m+1)$, $B((\al+1)/2)=\prod\Gamma((\al_j+1)/2)/\Gamma(\sum(\al_j+1)/2)$ is the multi-variable Beta function, and $d\sigma_m$ is the Riemannian density that $dx$ induces on the unit sphere $\bS^{m-1}\sub\bR^m$.
\end{lemma}
\begin{proof}
These are standard facts \cite[Section 1.8]{andrews},\cite[Page 13]{zhu-FT},\cite[Page 80]{folland}.
\end{proof}

\begin{proposition}\label{lemma2}
(a)
Given multi-index $\al\in\bN^m$, we have
\[
\om_1(\al)
=\left\|z^\al\right\|^2_{L^2_a(\Om_1)}
=\frac{\pi^m}{\prod p_j}\frac{B\left(\frac{\al+1}{p}\right)}{\left|\frac{\al+1}{p}\right|},
\]
where $\frac{\al+1}{p}:=\left(\frac{\al_1+1}{p_1},\ldots,\frac{\al_m+1}{p_m}\right)$.

(b)
Given multi-indices $\al\in\bN^m, \be\in\bN^n, \ldots$, we have
\begin{multline*}
\om_{2}(\al,\be,\ldots)
=\left\|z^\al w^\be\cdots\right\|^2_{L^2_a(\Om_2)}=\\
\frac{\pi^{m+n+\cdots}}{\prod p_j\prod q_k\cdots}
\frac{1}{ab\cdots}
B\left(\frac{\al+1}{p}\right)
B\left(\frac{\be+1}{q}\right)\cdots
\frac{B\left(\left|\frac{\al+1}{ap}\right|,\left|\frac{\be+1}{bq}\right|,\ldots\right)}{\left|\frac{\al+1}{ap}\right|+\left|\frac{\be+1}{bq}\right|+\cdots}.
\end{multline*}
\end{proposition}

\begin{proof}
(a)
Using polar coordinates $z_j=x_je^{\sqrt{-1}\theta_j}$ for each coordinate $z_j$ of $z=(z_1,\ldots,z_m)$, we have
\[
\om_{1}(\al)
=\int_{z\in\Om_1}x^{2\al}\prod x_jdx_jd\theta_j
=(2\pi)^m\int_{x\in\bR_+^m,\sum x_j^{2p_j}<1}x^{2\al+1} dx.
\]
After the change of variables $X_j:=x_j^{p_j}$, we have
\[
\om_{1}(\al)
=\frac{(2\pi)^m}{\prod p_j}\int_{X\in\bR_+^m,\sum X_j^{2}<1}X^{\frac{2\al+2}{p}-1} dX.
\]
We are done by Lemma \ref{lemma1}.

(b)
Using polar coordinates $z_j=x_je^{\sqrt{-1}\theta_j}, w_k=y_ke^{\sqrt{-1}\vp_k}, \ldots$, we have
\begin{multline*}
\om_{2}(\al,\be,\ldots)=\\
	(2\pi)^{m+n+\cdots}\int_{x\in\bR_+^m,y\in\bR_+^n,\ldots,\left(\sum x_j^{2p_j}\right)^{a}+\left(\sum y_k^{2q_k}\right)^{b}+\cdots<1}x^{2\al+1}y^{2\be+1}\cdots dx dy\cdots.
	\end{multline*}
After the change of variables $X_j:=x_j^{p}, Y_k:=y_k^q, \ldots$, we have
	\begin{multline*}
	\om_2(\al,\be,\ldots)=\\
	\frac{(2\pi)^{m+n+\cdots}}{\prod p_j\prod q_k\cdots}\int_{X\in\bR_+^m,Y\in\bR_+^n,\ldots,\left(\sum X_j^{2}\right)^{a}+\left(\sum Y_k^{2}\right)^{b}+\cdots<1}X^{\frac{2\al+2}{p}-1}Y^{\frac{2\be+2}{q}-1}\cdots dX dY\cdots.
	\end{multline*}
Now comes the trick we learned from \cite{dangelo}.
Changing to the spherical coordinates $X=r\xi, Y=s\eta, \ldots$, where $r, s,\ldots$ are positive reals and $\xi, \eta, \ldots$ live on unit spheres $\bS^{m-1}, \bS^{n-1}, \ldots$, respectively, we have
\begin{multline*}
\om_2(\al,\be,\ldots)=
\frac{(2\pi)^{m+n+\cdots}}{\prod p_j\prod q_k\cdots}
\int_{r,s,\ldots\in\bR_+,r^{2a}+s^{2b}+\cdots<1}
r^{\left|\frac{2\al+2}{p}\right|-1}s^{\left|\frac{2\be+2}{q}\right|-1}\cdots dr ds\cdots\times\\
\int_{\xi\in\bS^{m-1}_{+},\eta\in\bS^{n-1}_{+},\ldots}
\xi^{\frac{2\al+2}{p}-1}\eta^{\frac{2\be+2}{q}-1}\cdots d\sigma_m(\xi) d\sigma_n(\eta)\cdots,
\end{multline*}
where $S^{m-1}_+$ denotes $\bS^{m-1}\cap\bR^{m}_+$, and similarly for others.
The first integral is given by the first formula in Lemma \ref{lemma1} after the change of variables $R:=r^a, S:=s^b, \ldots$, and the second integral is given by the second formula in Lemma \ref{lemma1}.
\end{proof}

For later uses, we do the same computations in the more general context of weighted Bergman spaces.
Given a domain $\Om\sub\bC^m$ with smooth boundary, $L_{a,s}^2(\Omega)$, $s>-1$ denotes the weighted Bergman space consisting of all holomorphic functions $f$ on $\Omega$ such that $\int_{\Omega}|f(z)|^2 \rho(z)^{s}dV(z)<\infty$, where $\rho(z)$ is a positively signed smooth defining function for $\Omega$ and $dV$ is the Lebesgue measure.
For $\Om_1$ and $\Om_2$, we use the defining functions
\[
1-\sum|z_j|^{2p_j}\quad\text{and}\quad
1-\left(\sum\left|z_j\right|^{2p_j}\right)^{a}-\left(\sum\left|w_k\right|^{2q_k}\right)^{b}-\cdots,
\]
respectively.

\begin{proposition}\label{lemma22}
(a)
Given multi-index $\al\in\bN^m$, we have
\[
\om_{1,s}(\al)
	:=\left\|z^\al\right\|^2_{L^2_{a,s}(\Om_1)}
	=\frac{\pi^m}{\prod p_j}B\left(\frac{\al+1}{p}\right)\frac{\Gamma\left(\left|\frac{\al+1}{p}\right|\right) s!}{\Gamma\left(\left|\frac{\al+1}{p}\right|+s+1\right)}.
	\]

(b)
Given multi-indices $\al\in\bN^m, \be\in\bN^n, \ldots$, we have
\begin{multline*}
\om_{2,s}(\al,\be,\ldots)
	:=\left\|z^\al w^\be\cdots\right\|^2_{L^2_{a,s}(\Om_2)}=\\
	\frac{\pi^{m+n+\cdots}}{\prod p_j\prod q_k\cdots}\frac{1}{ab\cdots}
	B\left(\frac{\al+1}{p}\right)
	B\left(\frac{\be+1}{q}\right)\cdots
	B\left(\left|\frac{\al+1}{ap}\right|,\left|\frac{\be+1}{bq}\right|,\ldots\right)\times\\
	\frac{s!\Gamma\left(\left|\frac{\al+1}{ap}\right|+\left|\frac{\be+1}{bq}\right|+\cdots\right)}{\Gamma\left(s+1+\left|\frac{\al+1}{ap}\right|+\left|\frac{\be+1}{bq}\right|+\cdots\right)}.
	\end{multline*}
\end{proposition}

\begin{proof}
(a)
Similar to the proof of Proposition \ref{lemma2}.(a), we have
\[
\om_{1,s}(\al)
=\frac{(2\pi)^m}{\prod p_j}\int_{X\in\bR_+^m,\sum X_j^{2}<1}X^{\frac{2\al+2}{p}-1}\left(1-\sum X_j^{2}\right)^{s} dX.
\]
Changing to the spherical coordinates $X=r\xi$, $r>0$, $\xi\in\bS^{m-1}$, we have
\[
\om_{1,s}(\al)
=\frac{(2\pi)^m}{\prod p_j}\int_{\xi\in\bS_+^{m-1}}\xi^{\frac{2\al+2}{p}-1}d\sigma_m(\xi)\times
\int_{0}^{1}r^{\left|\frac{2\al+2}{p}\right|-1}\left(1-r^2\right)^{s} dr.
\]
The first integral is given by the second formula in Lemma \ref{lemma1}, and the second integral is given by the formula 
$\int_{0}^{1}t^{a-1}(1-t)^{b-1}dt=\Gamma(a)\Gamma(b)/\Gamma(a+b)$ after the change of variable $r^2=t$.

(b)
Similar to the proof of Proposition \ref{lemma2}.(b), we have
\begin{multline*}
\om_{2,s}(\al,\be,\ldots)=
\frac{(2\pi)^{m+n+\cdots}}{\prod p_j\prod q_k\cdots}
\int_{\xi\in\bS^{m-1}_{+},\eta\in\bS^{n-1}_{+},\ldots}
\xi^{\frac{2\al+2}{p}-1}\eta^{\frac{2\be+2}{q}-1}\cdots d\sigma_m(\xi) d\sigma_n(\eta)\cdots\times\\
\int_{r,s,\ldots\in\bR_+,r^{2a}+s^{2b}+\cdots<1}
r^{\left|\frac{2\al+2}{p}\right|-1}s^{\left|\frac{2\be+2}{q}\right|-1}\cdots dr ds\cdots\left(1-r^{2a}-s^{2b}-\cdots\right)^s.
\end{multline*}
The first integral is given by the second formula in Lemma \ref{lemma1}.
The second integral after the change of coordinates $R:=r^a, S:=s^b,\ldots$ becomes 
\[
\frac{1}{ab\cdots}\int_{R,S,\ldots\in\bR_+,R^{2}+S^{2}+\cdots<1}
R^{\left|\frac{2\al+2}{ap}\right|-1}S^{\left|\frac{2\be+2}{bq}\right|-1}\cdots \left(1-R^{2}-S^{2}-\cdots\right)^sdR dS\cdots.
\]
Changing to the spherical coordinates $(R,S,\ldots)=r\xi$, $r>0$, $\xi$ in the unit sphere, this latter integral equals an integral in the second formula in Lemma \ref{lemma1} multiplied by some integral of the form $\int_{0}^{1}t^{u-1}(1-t)^{v-1}dt=\Gamma(u)\Gamma(v)/\Gamma(u+v)$.
\end{proof}

\subsection{Some notations}
From now on, we are going to use the notation
\begin{equation}
z^{\fn}:=\frac{z_1^{n^1}\ldots z_m^{n^m}}{\sqrt{\omega_1(\fn)}},\quad
\fn=(n^1, \ldots, n^m)\in \bN^m \label{normalized}
\end{equation}
for the elements of the orthonormal basis of $L^2_a(\Om_1)$ derived in Section \ref{onb}.

Given a positive integer $q$, let $S_q(m)$ denote the set of all $q$-shuffles of the set $\{1,\ldots, m\}$, namely
\[
S_{q}(m):=\left\{\fj:=(j^1, \ldots, j^q)\in\bZ^q : 1\leq j^1< j^2< \cdots < j^q\leq m\right\}. 
\]
Whenever necessary, we identify shuffles in $S_q(m)$ with subsets of $\{1,\ldots,m\}$ of size $q$. 
This enables us to talk about the union, intersection, etc. of shuffles of $\{1,\ldots,m\}$ with themselves and with other subsets of $\{1,\ldots,m\}$.

\subsection{Boxes and their associated Hilbert modules}\label{2-box}
To each $\fj=(j^1, \ldots, j^q)\in S_q(m)$ and $\fb=(b^1, \ldots, b^q)\in \bN^q$ we associate the \textit{box}
\[
\BB^{\fb}_{\fj}:=\left\{(n^1,\ldots, n^m)\in \bN^m :  n^{j^i} \leq b^i\ \text{for} \ i=1,\ldots,q\right\},
\]
and to each box $\BB_{\fj}^{\fb}$ we associate the Hilbert space
\[
\mathcal{H}^{\fb}_{\fj}:=L^2_a(\Om_1)\ominus\overline{\left\lang z^{b^1+1}_{j^1}, \ldots, z^{b^q+1}_{j^q}\right\rang}
\]
consisting of all functions $X=\sum_{\fn\in\bN^m}  X_{\fn}z^\fn\in L^2_a(\Om_1)$ such that $X_{\fn}=0$ for every $\fn\in\bN^m\setminus\mathbf{B}^{\fb}_{\fj}$.
An element $X\in \mathcal{H}_{\fj}^{\fb}$ has the Taylor expansion $X=\sum X_{n^1\cdots n^m} z^{\fn}$ with summation over $n^{j^1}\leq b^1, \ldots , n^{j^q}\leq b^q$.
The general construction in Section \ref{introduction} about the orthogonal complements of polynomial ideals makes $\cH_{\fj}^{\fb}$ a Hilbert $A$-module.
($A$ denotes the ring of polynomials in $m$ variables.)
More explicitly, its fundamental tuple of Toeplitz operators is given by 
\[
T^{\fj,\fb}_{z_i}(z^{\fn}):=\left\{\begin{array}{ll} z_i z^{\fn},&\text{if}\ (n^1,\ldots,n^{i-1}, n^i+1,n^{i+1}, \ldots, n^m)\in \BB^{\fb}_{\fj},\\ 0,&
\text{otherwise}, \end{array}\right.,\quad
i=1,\ldots,m.
\]

In the next proposition we gather several facts about essential normality which will be used later.

\begin{proposition}[Arveson-Douglas]\label{fact}
(a)
Let $\Om$ be an open subset of $\bC^m$, 
$I\sub A$ be a homogeneous ideal, and $P, Q:=1-P$ be the orthogonal projections in $L^2(\Om)$ onto $\overline{I}$ and $I^{\perp}$, respectively.
Suppose that $L^2_a(\Om)$ is essentially normal.
Then $\overline{I}$ is essentially normal (module actions are given by restrictions of multiplications in $L^2_a(\Om)$) if and only if
$I^{\perp}$ is essentially normal, if and only if all $[M_{z_\al},P]$, $\al=1,\ldots,m$ are compact, if and only if all $PM_{z_\al}Q$ are compact, if and only if all $[M_{z_\al},Q]$ are compact, if and only if all $QM_{z_\al}^{\ast}P$ are compact.

	(b)
	Let $\cM$ and $\cN$ be isomorphic Hilbert $A$-modules.
	Then $\cM$ is essentially normal if and only of $\cN$ is; if so, then they represent the same odd $K$-homology class.

	(c)
	Let $\cM$ be an essentially normal Hilbert $A$-module, and $\cN\sub\cM$ be a submodule.
	Then $\cN$ is essentially normal if and only if the quotient module $\cM/\cN $ is.
	
	(d)
	Let $\Psi:\cA_1\ra\cA_2$ be a closed-range Hilbert $A$-module map between essentially normal Hilbert modules.
	Then the kernel and range of $\Psi$ are essentially normal.
\end{proposition}
\begin{proof}
	(a)
	Our reference is \cite[Theorem 4.3]{arveson-conjecture}.
	Recall that an operator $T$ is compact if and only if $T^{\ast}$ is compact, if and only if $TT^\ast$ is compact. 
	Let the module action of $p\in A$ on $L^2(\Om)$, $\overline{I}$ and $I^{\perp}$ be denoted by operators $M_p$, $R_p$ and $T_p$, respectively.
	For brevity, set $M_\al:=M_{z_{\al}}$, $R_\al:=R_{z_{\al}}$ and $T_\al:=T_{z_{\al}}$.
	The last four statements are easily seen to be equivalent.
	Here are the reasons. 
	Since $\overline{I}$ is invariant under $M_\al$, we have $PM_\al P=M_\al P$. Then 
	\[
	[M_\al,P]=M_\al P-PM_\al=PM_\al P-PM_\al=-PM_\al Q.
	\]
	The equality $P+Q=1$ gives $[M_\al,P]=-[M_\al,Q]$.
	Also note that $(PM_\al Q)^{\ast}=QM_\al^{\ast}P$.
	
For the rest, we need the assumption that $L^2(\Om)$ is essentially normal.
With an abuse of language, one says that, as mappings from $L^2(\Om)$ to $\overline{I}$, $R_\al P$ and $R_\be^{\ast} P$ equal $ P M_\al P = M_\al P$ and $PM_\be^\ast P$, respectively. Then 
	\begin{multline*}	
	[R_\al,R_\be^{\ast}]P
	=M_\al PM_\be^{\ast}P-PM_\be^{\ast}M_\al P
	\sim M_\al PM_\be^{\ast}P-PM_\al M_\be^{\ast}P
	=[M_\al,P]M_\be^{\ast}P\\
	=-P M_\al QM_\be^{\ast}P
	=-(P M_\al Q) (QM_\be^{\ast}P)
	=-(P M_\al Q) (PM_\be Q)^{\ast}
	=-[M_\al,P][M_\be,P]^{\ast},
	\end{multline*}
	where $\sim$ denotes equality modulo compacts.
	This identity shows that all $[R_\al,R_\be^{\ast}]$ are compact if and only if all $[M_\al,P]$ are.
	The rest of the proof is dual.
	As mappings from $L^2(\Om)$ to $I^{\perp}$, $T_\al Q$ and $T_\be^{\ast} Q$ equal $ Q M_\al Q$ and $QM_\be^\ast Q=M_\be^\ast Q$, respectively.
	We also have the identity
	\[
	[T_\al,T_\be^{\ast}]Q\sim [M_\be,Q]^{\ast}[M_\al,Q],
	\]
	which proves that all $[T_\al,T_\be^{\ast}]$ are compact if and only if all $[M_\al,Q]$ are.

	(b, c, d) Refer to 
	\cite[Proposition 4.4]{dty},
	\cite[Theorem 2.1]{douglas-monomial} and \cite[Theorem 2.2]{douglas-index}, respectively.
\end{proof}

\begin{lemma}\label{prop:ess-norm-box} 
	Each $\mathcal{H}_{\fj}^{ \fb}$ is essentially normal.
\end{lemma}

\begin{proof}
We first show that $L^2_a(\Om_1)$ is essentially normal. 
Let $M_{z_i}\in B(L^2_{a}(\Om_1))$, $i=1,\ldots,m$ be the multiplication by the coordinate function $z_i$.
Since these operators commutate with each other, according to the Fuglede-Putnam theorem, it suffices to verify that each $M_{z_i}$ is essentially normal.
Clearly  
\[
\left[M_{z_i},M_{z_i}^*\right](z^\fn)
=\lambda z^\fn,\quad
\forall \fn=(n_1,\ldots,n_m)\in\bN^m,
\]
where
\[
\lambda
=\lambda'-\lambda'',\quad
\lambda'
=\frac{\omega_1(n_1\cdots n_m)}{\omega_1(n_1\cdots n_i-1\cdots n_m)},\quad
\lambda''
=\frac{\omega_1(n_1\cdots n_i+1\cdots n_m)}{\omega_1(n_1\cdots n_m)}.
\]
We need to check that $\lambda\ra 0$ when the norm of $\fn$ (say the $l^1$ norm) tends to infinity.
By Proposition \ref{lemma2}.(a), we have
\[
\lambda'
=\frac{\Gamma\left(\frac{n_i+1}{p_i}\right)}{\Gamma\left(\frac{n_i}{p_i}\right)}
\frac{\Gamma\left(N+\frac{n_i}{p_i}\right)}{\Gamma\left(N+\frac{n_i+1}{p_i}\right)}
\frac{N+\frac{n_i}{p_i}}{N+\frac{n_i+1}{p_i}},
\]
where 
\[
N:=\sum_{l\neq i}\frac{n_l+1}{p_l}.
\]
Note that by Stirling's formula (or more strongly \cite{et}), $\Gamma(x+a)/\Gamma(x)\approx x^a$ as the real variable $x$ grows large.
Therefore, when $n_i$ is bounded and $N\ra\infty$, $\lambda'$ is dominated by $N^{-1/p_i}$, so $\lambda\ra 0$. 
On the other hand, when $n_i\ra\infty$, $\lambda'$ asymptotically behaves like
\[
\left(\frac{n_i}{p_i}\right)^{\frac{1}{p_i}}
\left(N+\frac{n_i}{p_i}\right)^{-\frac{1}{p_i}}
\approx\left(1-\frac{N}{\frac{n_i}{p_i}+N}\right)^{\frac{1}{p_i}}
\approx
1-\frac{N}{n_i+p_iN}.
\]
This shows that $\lambda=\lambda'-\lambda''\ra 0$ when $n_i\ra \infty$.
We have shown that $L^2_a(\Om_1)$ is essentially normal. 

Let $P$ be the orthogonal projection in $L^2_a(\Om_1)$ onto $\mathcal{H}_{\fj}^{\fb}$.
To prove our lemma, according to Proposition \ref{fact}.(a), it suffices to check that each $[M_{z_i},P]$ is compact. 
For each $\fn\in\BB^{\fb}_{\fj}$ we have
	\[
	PM_{z_i}(z^\fn)=\left\{ \begin{array}{ll} \sqrt{\frac{\omega_1(n_1\cdots n_i+1 \cdots n_m)}{\omega_1(n_1\cdots n_m)}} z^{n_1\cdots n_i+1 \cdots n_m},&\text{if}\ (n_1\cdots n_i+1\cdots n_m)\in  \BB^{\fb}_{\fj},\\ 0,&\text{otherwise},
	\end{array}\right.
	\]
	\[
	M_{z_i}P(z^{\fn})=\left\{ \begin{array}{ll} \sqrt{\frac{\omega_1(n_1\cdots n_i+1 \cdots n_m)}{\omega_1(n_1\cdots n_m)}} z^{n_1\cdots n_i+1\cdots n_m},&\text{if}\ (n_1\cdots n_i\cdots n_m)\in  \BB^{\fb}_{\fj},\\ 0,&\text{otherwise}.
	\end{array}\right.
	\]
	Note that the coefficients $\sqrt{\cdots}$ appear because of the normalization in definition (\ref{normalized}).
	Therefore
	\[
	\left[M_{z_i},P\right](z^{\fn})=\left\{\begin{array}{ll}
	\sqrt{\frac{\omega_1(n_1\cdots b_l+1\cdots n_m)}{\omega_1(n_1\cdots b_l\cdots n_m)}} z^{n_1\cdots n_i+1 \cdots n_m},&\text{if}\ (n_1\cdots n_i\cdots n_m)\in  \BB^{\fb}_{\fj}\  \text{and}\\
	&\exists l \ \text{such that} \ i=j_l, n_i=b_l,\\ 0,&\text{otherwise}.
	\end{array}\right.
	\]
We need to check that the ratio
\[
\rho:=\frac{\omega_1(n_1\cdots b_l+1\cdots n_m)}{\omega_1(n_1\cdots b_l\cdots n_m)},
	\]
with $l$ and $b_l$ fixed, approaches zero when the norm of $(n_1, \ldots, b_l, \ldots, n_m)$ tends to infinity. 
This was verified during the proof of the essential normality of $L^2_a(\Om_1)$.
%
%
This finishes the proof of our lemma for domains of type (\ref{egg}).
The proof for domains of type (\ref{egg2}) is completely similar, having the explicit formula for $\om_2(\al,\be,\ldots)$ (Proposition \ref{lemma2}.(b)) at hand.
\end{proof}

\begin{remark}\label{schatten}
With arguments similar to the ones in the proof of Lemma \ref{prop:ess-norm-box}, one can show that $L^2_a(\Om_1)$ is $p$-essentially normal if and only if $p$ is strictly larger than
\[
\max\left\{
m,
p_j(m-1): j=1,\ldots,m\right\}.
\]
The computations will be included in our forthcoming paper \cite{jabbari-schatten}.
It is worth pointing out that it is a new phenomenon that the $p$-essential normality of the Bergman module depends not only on the dimension of the domain but also on its geometry.
This phenomenon will also be explored in \cite{jabbari-schatten}.
\end{remark}

\subsection{The geometry of the Hilbert modules associated to boxes}\label{subsec:geometry}
Consider the Hilbert module $\cH_{\fj}^{\fb}$ associated to the box $\BB_{\fj}^{\fb}$. 
Set  
\[
\Om_{1,\fj}:=\left\{(z_1, \ldots, z_m)\in \Om_1 : \ z_{j^1}=\cdots =z_{j^q}=0\right\}. 
\]
Observe that $\Om_{1,\fj}$ is an egg domain of type (\ref{egg}) inside $\bC^{m-q}$.
Consider the Hilbert space 
\[
\widetilde{\mathcal{H}}_{\fj}^{\fb}
:=\bigoplus _{\substack{\mathfrak{i}=(i^1,\ldots,i^q)\in \bN^q \\ i^1\leq b^1, \ldots, i^q\leq b^q}} L^2_{a, \sum_{l=1}^q\frac{i^l+1}{p_{j^l}}}(\Om_{1,\fj}),
\]
and the map $R^{\fb}_{\fj}: \mathcal{H}_{\fj}^{\fb}\to \widetilde{\mathcal{H}}_{\fj}^{\fb}$ given by sending $X\in\mathcal{H}_{\fj}^{\fb}$ to 
{\small
\[
Y=\sum Y^{\mathfrak{i}},\quad
Y^{\mathfrak{i}}
=\left(\frac{\pi^q}{i!\prod_{l=1}^{q} p_{j^{l}}}\frac{\prod_{l=1}^q\Gamma \left(\frac{i^l+1}{p_{j^l}}\right)}{\left(\sum_{l=1}^q\frac{i^l+1}{p_{j^l}}\right)!}\right)^{\frac{1}{2}}\frac{\pa^{|i|} X}{\pa z_{j^1}^{i^1}\cdots \pa z_{j^q}^{i^q}}\bigg|_{\Om_{1,\fj}}\in L^2_{a, \sum_{l=1}^q\frac{i^l+1}{p_{j^l}}}(\Om_{1,\fj}).
\]
A straightforward computation with the orthonormal bases (Propositions \ref{lemma2} and \ref{lemma22}) shows that: \textit{$R^{\fb}_{\fj}$ is an isometric isomorphism of Hilbert spaces.}


Now consider the trivial vector bundle $E^\fb_{\fj}:=\bC^{(b^1+1)\cdots(b^q+1)}\times\Om_{1,\fj}$ over $\Om_{1,\fj}$, together with the standard frame $e_{\mathfrak{i}}$, $\mathfrak{i}=(i^1,\ldots,i^q)\in\bN^q$, $i^1\leq b^1, \ldots, i^q\leq b^q$, and equip it with the Hermitian structure
\[
\langle e_{\mathfrak{i}}, e_{\mathfrak{i}'}\rangle(z)=\left(1-\sum_{l=1}^q|z_l|^{2p_{j^l}}\right)^{\sum_{l=1}^q\frac{i^l+1}{p_{j^l}}}\delta_{\mathfrak{i}, \mathfrak{i}'},\quad
z\in \Om_{1,\fj},
\]
where $\delta$ is the Kronecker tensor.
This way, $\widetilde{\mathcal{H}}_{\fj}^{\fb}$ can be identified with the Bergman space of the $L^2$-holomorphic sections of $E^{\fb}_{\fj}$. 
Under the isomorphism $R^{\fb}_{\fj}$, one can identify the Toeplitz algebra generated by $T^{\fj, \fb}_{z_i}\in B(\cH^{\fb}_{\fj})$, $i=1, \ldots, m$ with the algebra generated by matrix-valued Toeplitz operators on the latter Bergman space of $L^2$-holomorphic sections of $E^{\fb}_{\fj}$.


\subsection{The construction of the resolution}\label{2-resolution}
This section constructs the resolution in Theorem \ref{theorem2}.
Let the ideal $I\sub A$ be generated by distinct monomials
\[
z^{\alpha_i},\quad
\alpha_i:=(\al_i^1,\ldots,\al_i^{m})\in\bN^m,\quad
i=1,\ldots, l.
\]
Let the \textit{complementary space} $\mathsf{C}(I)\sub\bN^m$ be the set of the exponents of those monomials which do not belong to $I$.
Note that the set of monomials belonging to $I$ is a basis of $I$ as a complex vector space \cite[Theorem 1.1.2]{hh}. 
Also note that a monomial $u$ belongs to $I$ if and only if there is a monomial $v$ such that $u=vz^{\alpha_i}$ for some $i=1, \ldots, l$ \cite[Proposition 1.1.5]{hh}.
In other words, $z_1^{n^1}\cdots z_m^{n^m}\in\mathsf{C}(I)$ if and only if for every $i=1,\ldots,l$ there exists $s_i\in\{1,\ldots,m\}$ such that $n^{s_i}< \alpha_i^{s_i}$. 
Consider the finite collection 
\[
S(\alpha_1, \ldots, \alpha_l):=\{1,\ldots,m\}^l
\]
of  $l$-tuples $\mathfrak{s}=(s_1, \ldots, s_l)$ of integers such that $1\leq s_i\leq m$ for every $i$. 
Given $\mathfrak{s}$, let $\fj_{\mathfrak{s}}$ be the shuffle associated to the set $\{s_1,\ldots,s_l\}$.
For each $j\in \fj_{\mathfrak{s}}$, let $b_j$ be the minimum of all $\alpha_i^{s_i}-1$, $i=1,\ldots,l$, such that $s_i=j$. 
Set $\fb_{\fs}:=(b_j)_{j\in\fj_{\fs}}$. 
The following symbolic logic computation shows that: \textit{$\mathsf{C}(I)$ is the union of boxes $\BB_{\fj_{\fs}}^{\fb_{\fs}}$, $\mathfrak{s}\in S(\alpha_1, \ldots, \alpha_l)$}.
\begin{align*}
z_1^{n^1}\cdots z_m^{n^m}\in\mathsf{C}(I)
&\leftrightarrow
\Big(n^1<\al_1^1\vee\cdots\vee n^m<\al_1^m\Big)\w\cdots
\w\Big(n^1<\al_l^1\vee\cdots\vee n^m<\al_l^m\Big)\\
&\leftrightarrow
\bigvee_{(s_1,...,s_l)\in\{1,\ldots,m\}^l}\Big(
n^{s_1}<\al_1^{s_1}\w\cdots\w
n^{s_l}<\al_l^{s_l}\Big).
\end{align*}



\subsubsection{The construction of modules $\cA_q$}\label{subsec:hilbert-mod}
From now on, fix a finite collection of boxes
\begin{equation}
\BB_{\fj_i}^{\fb_i},\quad
i=1,\ldots,k\label{some}
\end{equation}
such that their union equals $\mathsf{C}(I)$.
Given $I\sub\{ 1, \ldots, k\}$ (note that we are using the symbol $I$ for two purposes), let 
\[
\BB_{\fj_I}^{\fb_I}:=\bigcap_{i\in I} \BB_{\fj_i}^{\fb_i}
\]
denote the intersection of boxes $\BB_{\fj_i}^{\fb_i}$, $i\in I$.
(Note that the intersections of boxes are again boxes.)
Each box $\BB_{\fj_I}^{\fb_I}$ has a corresponding Hilbert module $\mathcal{H}_{\fj_I}^{\fb_I}$ as introduced in Section \ref{2-box}. 
For each $q=1,\ldots,k$, set

\[
\cA_q:=\bigoplus\limits_{I\in S_q(k)} \mathcal{H}_{\fj_I}^{\fb_I},\quad
\cA_0:=L^2_a(\Om_1).
\]
Note that each Hilbert space $\cA_q$ is equipped with a Hilbert $A$-module structure coming from the $A$-module structures on its direct summands. 
It is immediate from Lemma \ref{prop:ess-norm-box} that

\begin{proposition}\label{essnormalAq}
Each $\cA_q$ is essentially normal.
\end{proposition}

\subsubsection{The construction of maps $\Psi_q$}\label{subsec:morphism}
%
%

Thinking of the elements of $S_{q+1}(k)$ as the subsets $I_{q+1}\sub\{1,\ldots, k\}$ of size $q+1$, define the maps $f^i_{q+1}: S_{q+1}(k)\to S_{q}(k)$, $i=1, \ldots, q+1$ by setting $f^i_{q+1}(I_{q+1})$ to be the subset of $\{1, \ldots, k\}$ obtained by dropping the $i$-th smallest element in $I_{q+1}$. 
The map $\Psi_q:\cA_q\ra\cA_{q+1}$ is defined by sending $X=\sum_{I_q\in S_q(k)} X^{I_q}\in\cA_q$, $X^{I_q}\in  \mathcal{H}_{\fj_{I_q}}^{\fb_{I_q}}$ to $Y=\sum_{I_{q+1}\in S_{q+1}(k)} Y^{I_{q+1}}\in\cA_{q+1}$, $Y^{I_{q+1}}\in \mathcal{H}_{\fj_{I_{q+1}}}^{\fb_{I_{q+1}}}$, given by  
\[
\left(Y^{I_{q+1}}\right)_{\fn}=
\begin{cases}
\sum_{i=1}^{q+1}(-1)^{i-1} \left(X^{f^i_{q+1}(I_{q+1})}\right)_{\fn},&\fn\in \BB_{\fj_{I_{q+1}}}^{\fb_{I_{q+1}}},\\
0,&\text{otherwise}.
\end{cases}
\]

\begin{remark}
Similar to the explanation in Section \ref{subsec:geometry}, each Hilbert module $\cA_q$, $q=1,\ldots, k$ can be identified with  the Bergman space of the $L^2$-holomorphic sections of a Hermitian vector bundle on a disjoint union of subsets of $\Om_1$. Under this identification, the module morphisms $\Psi_q$, $q=0, \ldots, k-1$ can be realized as the restriction maps of jets of holomorphic sections to the subsets. 
Although this geometric picture is not used heavily in what follows, we believe that such an intuition will play a crucial role in the study of nonradical ideals beyond monomials.
\end{remark}

\section{The Proof of Theorem \ref{theorem2}}\label{proof}
In this section we prove Theorem \ref{theorem2}.
Again, we develop the details for a domain $\Om_1$ of type (\ref{egg}), and domains of type (\ref{egg2}) can be treated similarly.
\noindent
\subsection{The proof of Theorem \ref{theorem2}.(a)}
In this section we prove that the construction of Section \ref{2-resolution} is a resolution of Hilbert modules   asserted in Theorem \ref{theorem2}.(b).
This is an adjustment of the  proof of Theorem \ref{theorem1}, first appeared in \cite[Theorem 1.1]{djty}.

\begin{proposition}\label{prop:module-morphism}
Each $\Psi_q$ is a morphism of Hilbert $A$-modules.
\end{proposition}
\begin{proof} 
We first verify boundedness.
For each $X=\sum_{I_q \in S_q(k)} X^{I_q}\in\cA_q$, $X^{I_q}\in \mathcal{H}_{\fj_{I_q}} ^{\fb_{I_q}}$ we defined 
\[
\Psi_q(X)=\sum_{I'_{q+1}} Y^{I'_{q+1}},\ \ 
Y^{I'_{q+1}}\in \mathcal{H}_{\fj_{I'_{q+1}}} ^{\fb_{I'_{q+1}}},\ \ 
Y^{I'_{q+1}}_{\fn}
=\begin{cases}
\sum_{i=1}^{q+1}(-1)^{i-1} X^{f^i_{q+1}(I'_{q+1})}_{\fn},&\fn\in \BB_{\fj_{I'_{q+1}}} ^{\fb_{I'_{q+1}}},\\
0,&\text{otherwise}.
\end{cases}
\]
Therefore
\begin{align*}
\left\|\Psi_q(X)\right\|^2
&=
\sum_{I'_{q+1}}\sum_{\fn\in \BB_{\fj_{I'_{q+1}}} ^{\fb_{I'_{q+1}}}} \left|\sum_{i=1}^{q+1}(-1)^{i-1}  X^{f^i_{q+1}\left(I'_{q+1}\right)}_{\fn}\right|^2
\leq \sum_{I'_{q+1}}\sum_{\fn\in \BB_{\fj_{I'_{q+1}}} ^{\fb_{I'_{q+1}}}} (q+1)\left|X^{f^i_{q+1}(I'_{q+1})}_{\fn}\right|^2\\
&\leq \sum_{I'_{q+1}}\sum_{\fn\in \BB_{\fj_{I_{q}}} ^{\fb_{I_q}} } (q+1)\left|X^{I_{q}}_{\fn}\right|^2,\quad
	\text{since}\ \BB_{\fj_{I'_{q+1}}} ^{\fb_{I'_{q+1}}}\subseteq \BB_{\fj_{I_{q}}} ^{\fb_{I_q}}\\
&\leq
(k-q)(q+1) \sum_{I\in S_q(k)}\sum_{\fn\in \BB_{\fj_{I_{q}}} ^{\fb_{I_q}} } \left|X^{I_{q}}_{\fn}\right|^2
=(k-q)(q+1)\|X\|^2.
\end{align*}
The last inequality is because every $I_q$ is contained in at most $k-q$ number of $I'_{q+1}$.

Next, we prove that $\Psi_q$ commutes with the module actions.
For each $I\in S_q(k)$ and $X^I\in \mathcal{H}^{\fb_I} _{\fj_I}$, we defined 
{\small
\[
\Psi_q(X^I)
=\sum_{1\leq s\leq k,\  s\notin I} (-1)^{\text{sign}(I, s)} Y^{I\cup \{ s\}},\quad
Y^{I\cup \{s\}}\in \mathcal{H}^{\fb_{I\cup \{s\}}} _{\fj_{I\cup \{s\}}},\quad
Y^{I\cup \{s\}}_{\fn}
=\begin{cases}
X^I_{\fn},& \fn\in \BB^{\fb_{I\cup \{s\}}}_{\fj_{I\cup \{s\}}},\\
0,&\text{otherwise},
\end{cases}
\]}
where $s$ is the $\alpha$-th smallest number in $I\cup \{s\}$, and sign$(I,s)=\alpha-1$.
	
Each $z_p$ action on $\mathcal{H}_{\fj_I} ^{\fb_I}$ is implemented by
	\[
	T^{\fj_I, \fb_I}_{z_p}\left(X^I\right)_{n_1\cdots n_p+1 \cdots n_m}
	=\left\{\begin{array}{ll}\sqrt{\frac{\omega_1(n_1\cdots n_p+1\cdots n_m)}{\omega_1(n_1\cdots n_m)}} X^I_{n_1\cdots n_p\cdots n_m},& p\notin \fj_I,\\
	\sqrt{\frac{\omega_1(n_1\cdots n_p+1 \cdots n_m)}{\omega_1(n_1\cdots n_m)}} X^I_{n_1\cdots n_p\cdots n_m},& p=j^s\in \fj_I,\  n_p+1\leq b^s,\\
	0,&\text{otherwise}.
	\end{array}\right.
	\]
This shows that $T^{\fj_I, \fb_I}_{z_p}$ preserves the component $\mathcal{H}_{\fj_I} ^{\fb_I}$. Similarly, the $z_p$ action on $\mathcal{H}^{\fb_{I\cup \{s\}}}_{\fj_{I\cup \{s\}}}$ is realized by
	\begin{multline*}
	T^{\fj_{I\cup\{s\}}, \fb_{I\cup \{s\}}}_{z_p}\left(Y^{I\cup \{s\}}\right)_{n_1\cdots  n_p+1\cdots n_m}\\
	=\left\{\begin{array}{ll}\sqrt{\frac{\omega_1(n_1\cdots  n_p+1 \cdots n_m)}{\omega_1(n_1\cdots n_m)}} Y^{I\cup\{s\}}_{n_1\cdots n_p\cdots n_m},& p\notin \fj_I,\ \ p\ne s,\\
	\sqrt{\frac{\omega_1(n_1\cdots n_p+1\cdots n_m)}{\omega_1(n_1\cdots n_m)}} Y^{I\cup\{s\}}_{n_1\cdots n_p\cdots n_m},& p=j^t\in \fj_{I\cup \{s\}},\  n_p+1\leq b^t,\\
	0,&\text{otherwise}.
	\end{array}\right.
	\end{multline*}
It is straightforward to directly check that on each component $\mathcal{H}^{\fb_{I\cup \{s\}}}_{\fj_{I\cup \{s\}}}$ we have
	\[
	\left(\Psi_q\left(T^{\fj_I, \fb_I}_{z_p}\left(X^I\right)\right)\right)^{I\cup \{s\}}= T^{\fj_{I\cup \{s\}}, \fb_{I\cup \{s\}}}_{z_p}\left(\Psi_q\left( X^I\right)^{I\cup \{ s\}}\right),
	\]
and we are done.
\end{proof}

\begin{proposition}
$\overline{I}=\ker(\Psi_0)$.
\end{proposition}
\begin{proof}
If $f\in I$, then $f$ has no nonzero component in any of the boxes $\BB_{\fj_{\mathfrak{s}}}^{\fb_{\mathfrak{s}}}$, $\mathfrak{s}\in S(\alpha_1, \ldots, \alpha_l)$, hence $f\in \ker(\Psi_0)$.
This shows that $\overline{I}\sub\ker(\Psi_0)$.
For the other direction, assume $f=\sum_{\fn\in\bN^m} f_{\fn}z^{\fn}\in\ker(\Psi_0)$.
Since $\Psi_0(f)=0$, it follows that $f_{\fn}=0$ for every $i=1, \ldots, k$ and $\fn\in \BB_{\fj_i}^{\fb_i}$.
Let $f_M$, $M=1,2,\ldots$ be the truncation of the Taylor expansion of $f$ by requiring $n^1, \ldots, n^m<M$.
$f_M$ has no component in the  boxes $\BB_{\fj_1}^{\fb_1},\ldots, \BB_{\fj_k}^{\fb_k}$, hence $f_M\in I$.
Therefore, $f=\lim f_M\in\overline{I}$.
\end{proof}

\begin{proposition}\label{prop:kernel} $\operatorname{Im}(\Psi_{q-1})\subseteq \ker(\Psi_q)$ for every $q=1,\ldots,k$.
\end{proposition}
\begin{proof}
For each $I\in S_{q-1}(k)$ and $X^I\in \mathcal{H}^{\fb_I}_{\fj_I}$, the image of $X^I$ under $\Psi_{q-1}$ is of the form 
	\[
	\sum_{1\leq s\leq k,\ s\notin I} (-1)^{\text{sign}(I, s)} Y^{I\cup \{ s\}},
	\]
	where $Y^{I\cup \{s\}}\in \mathcal{H}^{\fb_{I\cup \{s\}}} _{\fj_{I\cup \{s\}}}$, $s$ is the $\alpha$-th smallest number in $I\cup \{s\}$, sign$(I,s)=\alpha-1$, and the function $Y^{I\cup \{s\}}$ is given by 
	\[
	Y^{I\cup \{s\}}_{\fn}=\left\{\begin{array}{ll} X^I_{\fn},& \fn\in \BB^{\fb_{I\cup \{s\}}}_{\fj_{I\cup \{s\}}},\\ 
	0,&\text{otherwise}.
	\end{array}\right.
	\]
Similarly, the image of $Y^{I\cup \{s\}}$ under $\Psi_q$ is of the form 
	\[
	\sum_{1\leq t\leq k,\ t\notin I\cup \{s\}} (-1)^{\text{sign}(I\cup \{s\},t)} Z^{I\cup \{ s,t\}},
	\]
	where $Z^{I\cup \{s,t\}}\in \mathcal{H}^{\fb_{I\cup \{s,t\}}} _{\fj_{I\cup \{s,t\}}}$, $t$ is the $\beta$-th smallest number in $I\cup \{s, t\}$, sign$(I\cup\{s\}, t)=\beta-1$, and the function $Z^{I\cup \{s,t\}}$ is given by 
	\[
	Z^{I\cup \{s,t\}}_{\fn}=\left\{\begin{array}{ll} Y^{I\cup \{s\}}_{\fn},& \fn\in \BB^{\fb_{I\cup \{s,t\}}}_{\fj_{I\cup \{s,t\}}},\\ 
	0,&\text{otherwise}.
	\end{array}\right.
	\]
	Therefore
	\begin{multline*}
	\Psi_q\left(\Psi_{q-1}\left(X^I\right)\right)
=\sum_{1\leq s\ne t\leq k, \ s,t\notin I} (-1)^{\text{sign}(I, s)+\text{sign}(I\cup \{s\}, t)} Z^{I\cup \{ s,t\}}\\
	=\sum_{1\leq s<t\leq k, \ s,t\notin I} \Big((-1)^{\text{sign}(I, s)+\text{sign}(I\cup \{s\}, t)}+(-1)^{\text{sign}(I,t)+\text{sign}(I\cup \{t\}, s)} \Big)Z^{I\cup \{s,t\}}.
	\end{multline*}
Every summand in the latter sum vanishes because 
\[
\text{sign}(I, s)=\text{sign}(I\cup \{t\}, s),\quad
\text{sign}(I\cup \{s\}, t)=\text{sign}(I,t)+1
\]
when $s<t$.
\end{proof}

\begin{proposition}\label{special}
(a)
$\im(\Psi_0)\supseteq\ker(\Psi_1)$.

(b)
$\operatorname{Im}(\Psi_{q-1})\supseteq \ker(\Psi_q)$ for every $q=1, \ldots, k$.
\end{proposition}
\begin{proof}
(a)
Assume $X:=(X^1,\ldots,X^p)\in\mathrm{ker}(\Psi_1)$.
Consider $\xi\in\cA_0$ given by
	\[
	\xi_{\fn}:=\begin{cases}
	X^s_{\fn},&\text{there is}\ s\ \text{such that} \ \fn\in\BB_{\fj_s}^{\fb_s},\\
	0,&\text{otherwise}.
	\end{cases}
	\]
This is well-defined because $\Psi_1(\xi)=0$.
Note that $\xi\in\cA_0$ because $\|\xi\|^2=\|X^1\|^2+\cdots+\|X^p\|^2$.
Clearly, $\Psi_0(\xi)=X$.

(b)
We apply induction on $k$.  
When $k=1$, the map $\Psi_0: \cA_0\to \cA_1$ is surjective because computing with the orthonormal basis shows that $\cA_1$ can be identified with a closed subspace of $\cA_0=L^2_a(\Om_1)$, with $\Psi_0$ being the corresponding orthogonal projection.
Assuming
	\[
	\operatorname{Im}(\Psi_{q-1})\supseteq \ker(\Psi_q),\quad  q=1, \ldots, k,\quad 1\leq k<p,
	\]
we prove the statement for $k=p$. 
The case $q=1$ is proved in (a), so from now on we assume $2\leq q\leq k$.

Consider the following two collections of $p-1$ boxes:
	
\begin{itemize}
\item 
The first $p-1$ boxes: $\BB_{\fj_1}^{\fb_1}, \ldots, \BB_{\fj_{p-1}}^{\fb_{p-1}}$.

\noindent
Applying the construction in Section \ref{2-resolution} to these boxes, we get the Hilbert modules $\cA^1_s$ together with the Hilbert module maps $\Psi^1_s: \cA^1_s\to \cA^1_{s+1}$, $s=1,\ldots, p-2$. Set $\cA^1_p:=\{0\}$ and $\Psi^1_{p-1}:=0$.
		
\item 
The intersection of the first $p-1$ boxes with the last one: $\BB_{\fj_{1p}}^{\fb_{1p}}, \ldots, \BB_{\fj_{p-1p}}^{\fb_{p-1p}}$.

\noindent
Applying the construction in Section \ref{2-resolution} to these boxes, we get the Hilbert modules $\cA^2_s$ together with the Hilbert module maps $\Psi^2_s: \cA^2_s\to \cA^2_{s+1}$, $s=1,\ldots, p-2$. Set $\cA^2_p:=\{0\}$ and $\Psi^2_{p-1}:=0$.
\end{itemize}

By the induction assumption we have 
	\[
	\operatorname{Im}\left(\Psi^1_{q-1}\right)\supseteq \ker\left(\Psi^1_q\right),\quad \operatorname{Im}\left(\Psi^2_{q-1}\right)\supseteq \ker\left(\Psi^2_q\right),\quad q=1, \ldots, p-1.
	\]
	
	Define a map $\Phi_s: \cA^1_{s}\to \cA^2_{s}$ by 
	\[
	\Phi_s(X^I)=Y^{I\cup \{p\}},\quad I\in S_s(p-1),
	\]
	where $Y^{I\cup \{p\}}$ denotes the component corresponding to the intersection of the boxes $\BB_{\fj_{i_1p}}^{\fb_{i_1p}}$, $\ldots$, $\BB_{\fj_{i_{s}p}}^{\fb_{i_sp}}$, given by
	\[
	Y^{I\cup \{p\}}_{\fn}:=\left\{\begin{array}{ll} (-1)^{t}X^I_{\fn},& \fn\in \BB_{\fj_{I\cup\{p\}}} ^{\fb_{I\cup \{p\}}},\\ 0,&\text{otherwise}.\end{array}\right.
	\]
Similar to the proof of Proposition \ref{prop:module-morphism}, $\Phi_s$ is an $A$-module map.
Furthermore, we can easily check that
	\begin{itemize}
		\item 
		$\cA_q=\cA^1_{q}\oplus \cA^2_{q-1}$ for $q=2,\ldots, p$.
		
		\item 
		$\Psi_q=\left(\begin{array}{cc} \Psi^1_q&0\\ \Phi_q& \Psi^2_{q-1}\end{array}\right)$ for $q=2, \ldots, p-1$.  
	\end{itemize}
These identifications are used below to prove that $\text{Im}(\Psi_{q-1})\supseteq \ker(\Psi_{q})$.
We split the proof into three cases. 
	\begin{enumerate}[leftmargin=*]
		\item
		$q=2$.

\noindent
Suppose $(X_1, X_2)\in \cA^1_2\oplus \cA^2_1=\cA_2$ is in $\ker(\Psi_2)$. 
By the identification above for $\Psi_q$, we have 
		\[
		\Psi^1_2(X_1)=0,\qquad \Phi_2(X_1)+\Psi^2_1(X_2)=0. 
		\]
By the induction assumption, we have $\ker(\Psi^1_2)\subseteq \text{Im}(\Psi^1_1)$, so there exists $Y_1\in \cA^1_1$ such that $\Psi^1_1(Y_1)=X_1$.  
By Proposition \ref{prop:kernel}, for the morphism $\Psi_\bullet$, we have
		{\small 
			\begin{align*}
			(0,0)
			&=
			\Psi_2\big(\Psi_1(Y_1,0)\big)=\Psi_2\big(\Psi^1_1(Y_1), \Phi_1(Y_1)\big)\\
			&=
			\Big(\Psi^1_2\big(\Psi^1_1(Y_1)\big), \Phi_2\big(\Psi^1_1(Y_1)\big)+\Psi^2_1\big(\Phi_1(Y_1)\big)\Big),\quad \Psi^1_1(Y_1)=X_1, \ \Psi^1_2\big(\Psi^1_1(Y_1)\big)=0\\
			&=\Big(0, \Phi_2(X_1)+\Psi^2_1\big(\Phi_1(Y_1)\big) \Big).
			\end{align*}
}Therefore, $\Phi_2(X_1)+\Psi^2_1\big(\Phi_1(Y_1)\big)=0$. 
Setting $X_2':=X_2-\Phi_1(Y_1)$, we have
\[
\Psi^2_1(X_2')
=\Psi^2_1(X_2)-\Psi^2_1(\Phi_1(Y_1))=\Psi^2_1(X_2)+\Phi_2(X_1)=0,
\] 
because $0=\Psi_2(X_1, X_2)=(\Psi^1_2(X_1), \Phi_2(X_1)+\Psi^2_1(X_2))$.
Since $\Psi^2_1(X_2')=0$, it follows that the following assignment is well-defined:
		\[
		(Y_2)_{\fn}:=\left\{\begin{array}{ll}({X'_2}^{ip})_{\fn},&\fn\in \BB^{\fb_{ip}}_{\fj_{ip}}\ \text{for some }\ i=1,..., p-1,\\ 0,&\text{otherwise}. \end{array}\right.
		\]
Arguments similar to the proof of Proposition \ref{special} show that this assignment gives $Y_2\in \mathcal{H}_{\fj_p}^{\fb_p}$ such that $\Psi^2_0(Y_2)=X_2'$. 
In summary, we have found $(Y_1, Y_2)\in \cA_1=\cA^1_1\oplus\mathcal{H}_{\fj_p}^{\fb_p}$ which satisfies 
		\[
		\Psi_1(Y_1, Y_2)=\left(\Psi^1_1(Y_1), \Phi_1(Y_1)+\Psi_0^2(Y_2)\right)=\left(X_1, \Phi_1(Y_1)+X_2'\right)=(X_1, X_2).
		\]

		\item
		$q=3,\ldots, p-1$.

\noindent
Suppose $(X_1, X_2)\in \cA^1_{q}\oplus \cA^2_{q-1}=\cA_q$ is in $\ker(\Psi_q)$.  
		By the identification above for $\Psi_q$, we have 
		\[
		\Psi^1_q(X_1)=0,\qquad \Phi_q(X_1)+\Psi^2_{q-1}(X_2)=0. 
		\]
		Since $\operatorname{Im}(\Psi^1_{q-1})\supseteq \ker(\Psi^1_q)$, there exists $Y_1\in \cA^1_{q-1}$ such that $X_1=\Psi^1_{q-1}(Y_1)$. 
		Since $\Psi_q(\Psi_{q-1}(Y_1, 0))=0$, it follows that $\Phi_q(X_1)+\Psi^2_{q-1}(\Phi_{q-1}(Y_1))=0$.
		Therefore
		\[
		\Psi^2_{q-1}\left(X_2-\Phi_{q-1}(Y_1)\right)=0.
		\]
				Since $\operatorname{Im}\left(\Psi^2_{q-2}\right)\supseteq \ker\left(\Psi^2_{q-1}\right)$, there exists $Y_2\in \cA^2_{q-2}$ such that $\Psi^2_{q-2}(Y_2)=X_2-\Phi_{q-1}(Y_1)$.
In summary, we have found $(Y_1, Y_2)\in \cA_q$ which satisfies 
		\[
		\Psi_{q-1}(Y_1, Y_2)=\left(\Psi^1_{q-1}(Y_1), \Phi_{q-1}(Y_1)+\Psi_{q-2}(Y_2)\right)=(X_1, X_2). 
		\]

		\item
		$q=p$.
		
\noindent
Since $\Psi^2_{p-2}: \cA^2_{p-2}\to \cA^2_{p-1}$ is surjective, it follows that 
\[
\Psi_{p-1}: \left(\cA_{p-1}=\cA^1_{p-1}\oplus \cA^2_{p-2}\right)\ra\left(\cA_{p}=\cA^2_{p-1}\right)
\]is also surjective. 
	\end{enumerate}
	All cases are exhausted.
\end{proof}

\subsection{The proof of Theorem \ref{theorem2}.(b)}
To deduce the index formula in Theorem \ref{theorem2}.(b) from the resolution in Theorem \ref{theorem2}.(a), we need the following proposition.

\begin{proposition}\label{prop:isometry}
Let $0\ra \cM_1\ra\cM_2\ra\cM_3\ra 0$ be a short exact sequence of essentially normal Hilbert $A$-modules and Hilbert $A$-module maps between them.
Suppose that the essential spectra of $\cM_i$, $i=1,2,3$ is contained in $\overline{\Om_1}$, and let $\al_i: C(\overline{\Om_1})\to Q(\cM_i)$ be the $\ast$-representation of $C(\overline{\Om_1})$ on the Calkin algebra $Q(\cM_i)=B(\cM_i)/\fK(\cM_i)$ induced by the essential normality of $\cM_i$.

(a)
There are co-isometries $U: \cM_2\to \cM_1$ and $V: \cM_2\to \cM_3$ such that $UV^*=0=VU^*$, $U^*U+V^*V=1$, and they commute with $A$-module structures up to compact operators in the sense that $[U]\alpha_2 [U]^*=\alpha_1$ and $[V]\alpha_2[V]^*=\alpha_3$, where $\al_i(p)=[T^i_p]\in Q(\cM_i)$, $p\in A$ is the equivalence class of the multiplication operator $T^i_p\in B(\cM_i)$.

(b)
We have $[\alpha_2]=[\alpha_1]+[\alpha_3]$ in $K_1\left(\es^2\right)$, where $[\alpha_1]$ and $[\alpha_3]$ are identified as classes in $K_1(\es^2)$ by the co-isometries $U$ and $V$. 
\end{proposition}

\begin{proof} 
(a)
\cite[Proposition 3.8]{djty}. 

(b)
Set $\es^i:=\es(\cM_i)$. 
The representation $\al_i$ factors through $\ast$-monomorphism $C(\es^i)\ra Q(\cM_i)$. 
We have $\al_1=[U]\al_2[U]^*$ by (a).
The composition of $[U]\al_2[U]^*$ with $\al_1^{-1}$ is a $\ast$-homomorphism $C(\es^2)\ra C(\es^1)$, and this induces a natural map $\es^1\ra\es^2$. 
Similarly, we have a natural map $\es^3\ra\es^2$. 
Therefore, $\al_1$ and $\al_3$ induce classes $[\al_1]$ and $[\al_3]$ in $K_1\left(\es^2\right)$ by the functoriality of $K_1$.
Putting all equations
\[
UU^*=1=VV^*,\quad
UV^*=0=VU^*,\quad
U^*U+V^*V=1,
\]
\[
[U]\alpha_2 [U]^*=\alpha_1,\quad
[V]\alpha_2[V]^*=\alpha_3,
\]
together, we deduce that $[\alpha_2]=[\alpha_1]+[\alpha_3]$.
\end{proof}

\subsubsection*{The proof of Theorem \ref{theorem2}.(b)}
The idea is to decompose the resolution of $\overline{I}$ in Theorem \ref{theorem2}.(a) into short exact sequences and then apply Proposition \ref{prop:isometry}.(b).
The details follow.
Consider $\cA_q^{-}:=\operatorname{Im}(\Psi_{q-1})=\operatorname{ker}(\Psi_q)$ as a closed subspace of $\cA_q$.
Note that $\cA_k^-=\cA_k$ because $\Psi_{k-1}$ is surjective.
The morphism $\Psi_q: \cA_q\to \cA_{q+1}$ of Hilbert modules induces the short exact sequence
\begin{equation}
0\ra \cA_{q}^-\hookrightarrow \cA_q\xrightarrow{\Psi_q} \cA_{q+1}^{-}\ra 0,\quad
q=1, \ldots, k-1, \label{SESS}
\end{equation}
which, according to Propositions \ref{fact} and \ref{essnormalAq}, implies that $\cA_q^{-}$ is essentially normal.
Set $\es^{q}:=\es(\cA_q)$, and let $\alpha_q$ (respectively $\alpha_q^-$) be the $\ast$-monomorphism $C(\es^q)\ra Q(\cA_q)$ (respectively $C(\es^{q-})\ra Q(\cA_q^-)$) induced by essential normality.
Note that the essential spectra of all terms in exact sequence (\ref{SESS}) is contained in $\overline{\Om_1}$.
By Proposition \ref{prop:isometry}.(b), we have $[\alpha_q]=[\alpha_q^-]+[\alpha_{q+1}^-]$ in $K_1\left(\es^q\right)$ for every $q=1,\ldots,k-1$.
These formulas for $q=k-1$ and $q=k-2$ give
\[
[\alpha_{k-1}]=[\alpha_{k-1}^-]+[\alpha_k]\in K_1\left(\es^{k-1}\right),\quad
[\alpha_{k-2}]=[\alpha_{k-2}^-]+[\alpha_{k-1}^-]\in K_1\left(\es^{k-2}\right).
\]
Pushing forward these equations into $K_1\left(\es^{k-1}\cup \es^{k-2}\right)$ by inclusion maps $\es^{k-1}, \es^{k-2}\hookrightarrow \es^{k-1}\cup \es^{k-2} $, gives $[\alpha_{k-1}]+[\alpha_{k-2}^-]=[\alpha_k]+[\alpha_{k-2}]$.
Continuing this argument, we have
\begin{equation}
[\alpha_{1}^-]=[\alpha_1]-[\alpha_2]+\ldots+(-1)^{k-1}[\alpha_{k}]\quad
\mathrm{in}\quad
K_1(\es^1\cup \cdots \cup \es^k).\label{hope}
\end{equation} 

On the other hand, the short exact sequence
\[
0\ra \overline{I}\ra L^2_a(\Om_1)\ra \cA_1^-\ra 0
\]
establishes a natural Hilbert module isomorphism between $\cA_1^-$ and $L^2_a(\Om_1)/\overline{I}\cong I^{\perp}$, hence $\tau_I:=[I^{\perp}]=[\al_1^{-}]$ by Proposition \ref{fact}.(b). 
This, together with (\ref{hope}), gives the index formula in Theorem \ref{theorem2}.(b). Q.E.D.

\end{document}